\documentclass[12pt]{article}
\usepackage{amsfonts}

\usepackage[all]{xy}

\def\squarebox#1{\hbox to #1{\hfill\vbox to #1{\vfill}}}
\newcommand{\qed}{\hspace*{\fill}
\vbox{\hrule\hbox{\vrule\squarebox{.667em}\vrule}\hrule}\smallskip}
\newtheorem{teorema}{Theorem}[section]
\newtheorem{lema}[teorema]{Lemma}
\newtheorem{corolario}[teorema]{Corollary}
\newtheorem{proposicao}[teorema]{Proposition}

\newtheorem{remark}[teorema]{Remark}
\newenvironment{proof}{\noindent {\bf Proof:}}{\hfill $\qed $ \newline}

\begin{document}

\title{A note on periodic differential equations}
\author{Mauro Patr\~{a}o}
\maketitle

\begin{abstract}
Let $F$ be a Banach space and $L(F)$ be the set of all its bounded
linear operators. In this note, we are interested in the
asymptotic behavior (recurrence and chain recurrence) of the
solution of the following initial value problem
\begin{equation}\label{eqlinear}
x'(t) = X(t)x(t), \qquad x(0) = x,
\end{equation}
where $x \in F$ and the map $t \mapsto X(t) \in L(F)$ is a
$T$-periodic continuous curve. This asymptotic behavior is related
to the asymptotic behavior of the discrete-time flow on $F$
generated by the invertible operator $g \in L(F)$ given by the
associated fundamental solution at time $T$.
\end{abstract}

\noindent \textit{AMS 2000 subject classification}: Primary:
37B35, 34G10. Secondary: 34A30.

\noindent \textit{Key words:} Recurrence, chain transitivity,
stable sets, Banach space.

\section{Introduction}

Denote by $|\cdot|$ the norm of the Banach space $F$. The
canonical operator norm in $L(F)$ is also denoted by $|\cdot|$. It
is well known that the solution of (\ref{eqlinear}) is given by
$x(t) = g(t)x$, where $t \mapsto g(t) \in L(F)$ is a
differentiable curve, called \emph{the associated fundamental
solution}, satisfying the following initial value problem
\[
g'(t) = X(t)g(t), \qquad g(0) = I,
\]
where $I$ is the identity operator. The existence and the
uniqueness of the fundamental solution defined for all $t \in
\mathbb{R}$ follow from the Picard method, since $X(t)$ is
uniformly bounded and $L(F)$ is a Banach space. We can use this
fact to show that the inverse of $g(s)$ is given by $h(-s)$, where
$t \mapsto h(t)$ is the fundamental solution associated to $t
\mapsto X(t+s)$. We can assume that the period $T = 1$, since
otherwise we can replace $X(t)$ by $(1/T)X(t/T)$ whose associated
fundamental solution is given by $h(t) = g(t/T)$. The existence
and the uniqueness of the fundamental solution together with the
periodicity of $X(t)$ can be used to prove that
\begin{equation}\label{eqgtgn}
g(t+n) = g(t)g^n,
\end{equation}
for all $t \in \mathbb{R}$ and all $n \in \mathbb{N}$, where $g =
g(1)$.

We denote by $S^1$ the quotient $\mathbb{R}/\mathbb{Z}$ and by
$\overline{s}$ the class $s + \mathbb{Z}$, where $s \in
\mathbb{R}$. Using equation (\ref{eqgtgn}), the so called
\emph{associated skew-product flow in $S^1 \times F$}, defined by
\[
\phi^t(\overline{s},x) = (\overline{s+t},g(t+s)g(s)^{-1}x),
\]
is a well defined continuous map such that $\phi^t(\overline{0},x)
= (\overline{t},x(t))$. When $F$ is finite dimensional, Floquet
theorem (see e.g \cite{chi}) ensures that there exists $X \in
L(F)$ such that $g^2 = \mbox{e}^X$. In this case, we have that the
associated skew-product is conjugated with the flow given by
\[
\psi^t(\overline{s},x) = (\overline{s+t},\mbox{e}^{tX}x).
\]
On the other hand, when $F$ is infinite dimensional, there is not
in general any $X \in L(F)$ such that $g^n = \mbox{e}^X$, for some
$n \in \mathbb{N}$. Thus we can not always analyze the asymptotic
behavior of the solution $x(t) = g(t)x$ by considering the
asymptotic behavior of an exponential flow in the same way which
is done by the Floquet theory.

In this note, we provide an approach which overcome this
difficult, relating the asymptotic behavior of the associated
skew-product flow with the asymptotic behavior (recurrence and
chain recurrence) of the discrete-time flow on $F$ generated by
the invertible operator $g$. The presented results can be viewed
as a sort of extension to the infinite dimensional situation of
the Selgrade Theorem (see \cite{sel}), which deals with linear
flows on finite dimensional vector bundles over chain-transitive
flows on compact base spaces. As we observe in the end of this
note, all the results remain true if we replace the fiber given by
Banach space $F$ by any compact metrizable fiber where the
topological group of the bounded invertible linear operators acts
continuously. In this way, the results can also be viewed as a
sort extension of the results presented in \cite{bsm} and
\cite{msm}, which deal with flows of automorphisms of flag bundles
over chain-transitive flows on compact base spaces.

\section{Some technical results}

In this section, we prove some technical results concerning the
fundamental solution and the associated skew-product flow. We
observe that, by the definitions, it follows that
\begin{equation}\label{eqphig}
\phi^t(\overline{s},g(s)x) = (\overline{s+t},g(s+t)x).
\end{equation}
for all $s,t \in \mathbb{R}$ and $x \in F$.

\begin{lema}\label{lemphisgsx}
We have that
\[
S^1 \times F = \{(\overline{s},g(s)x) : \, s \in [0, 1), x \in
F\}.
\]
Furthermore, for each $s \in [0, 1)$ and all $t \geq 0$, there
exist $\tau \in (-1,1)$ and $n \in \mathbb{N}$ such that
\[
t = \tau + n \qquad \mbox{and} \qquad s+\tau \in [0,1).
\]
In this case, we have that
\begin{equation}\label{eqphisgsx}
\phi^t(\overline{s},g(s)x) = (\overline{s+\tau}, g(s+\tau)g^n x).
\end{equation}
\end{lema}
\begin{proof}
For each $r \in \mathbb{R}$, there exists a unique $s \in [0,1)$
such that $\overline{r} = \overline{s}$. Thus, for each $y \in F$,
putting $x = g(s)^{-1}y$, we have that $(\overline{r},y) =
(\overline{s},g(s)x)$, showing the first assertion. For the second
assertion, let $n$ the greatest natural number such that $n \leq s
+ t$. Since $n \leq s + t < (n+1)$, setting $\tau = t - n$, it
follows that $s + \tau \in [0,1)$. Thus $\tau \in (-1,1)$, since
$s \in [0, 1)$. For the last assertion, by equations
(\ref{eqphig}) and (\ref{eqgtgn}), it follows that
\[
\phi^t(\overline{s},g(s)x) = (\overline{s+\tau+n},g(s+\tau+n)x) =
(\overline{s+\tau}, g(s+\tau)g^n x).
\]
\end{proof}

We denote by $d$ the metric induced in $F$ by the norm $|\cdot|$
and consider the following metric in $S^1 \times F$ given by
\[
\overline{d}((\overline{s},x),(\overline{r},y)) =
\min\{|s-r|,1-|s-r|\} + d(x,y),
\]
where $s,r \in[0,1)$ and $x,y \in F$.

\begin{lema}\label{lemC}
There exists a constant $C \geq 1$ such that
\begin{equation}\label{eqlipg}
d(g(t)x, g(t)y) \leq Cd(x, y)
\end{equation}
for each $t \in [0,1)$ and every $x,y \in F$. We also have that
\begin{equation}\label{eqlipphi}
\overline{d}(\phi^t(\overline{s},x), \phi^t(\overline{r},y)) \leq
C\overline{d}((\overline{s},x), (\overline{r},y))
\end{equation}
for each $r,s,t \in [0,1)$ and every $x,y \in F$.
\end{lema}
\begin{proof}
Since $t \mapsto |g(t)|$ is a real continuous function, it is
bounded by a constant $C \geq 1$ in the compact interval $[0,1]$.
Hence we have that
\[
d(g(t)x, g(t)y) = |g(t)(x-y)| \leq |g(t)||x-y| \leq Cd(x, y).
\]
The inequality (\ref{eqlipphi}) follows immediately from the
inequality (\ref{eqlipg}) and the definitions of $\overline{d}$
and of $\phi^t$.
\end{proof}

Given a topological space $U$, we denote by $\mathcal{P}(U)$ the
set of all positive continuous functions whose domain is the whole
space $U$. We also need the following lemma.

\begin{lema}\label{lemc}
There exists $c \in \mathcal{P}(F)$ with $c \geq 1$ such that
\begin{equation}\label{eqcx}
d(x,y) \leq c(x)(|s-r|+d(g(s)x, g(r)y))
\end{equation}
for all $r,s \in [-2,2]$ and every $x,y \in F$.
\end{lema}
\begin{proof}
We have that $t \mapsto g(t)^{-1}$ is a differentiable map, since
it is a composition of the differentiable maps $t \mapsto g(t)$
and $h \mapsto h^{-1}$. By the mean value inequality, there exists
$B \geq 1$ such that $|g(s)^{-1} - g(r)^{-1}| \leq B|s - r|$, for
all $r,s \in [-2,2]$. We also have that there exists $D \geq 1$
such that $|g(r)|, |g(r)^{-1}| \leq D$, for all $r \in [-2,2]$.
Defining $c(x) = \max\{BD|x|, D\}$, we have that $c \in
\mathcal{P}(F)$ with $c \geq 1$ and that
\begin{eqnarray*}
d(x,y) & = & |g(s)^{-1}g(s)x - g(r)^{-1}g(r)y| \\
& \leq & |g(s)^{-1}g(s)x - g(r)^{-1}g(s)x| + |g(r)^{-1}g(s)x - g(r)^{-1}g(r)y| \\
& \leq & |g(s)^{-1} - g(r)^{-1}||g(s)||x| + |g(r)^{-1}||g(s)x - g(r)y| \\
& \leq & BD|x||s - r| + D|g(s)x - g(r)y| \\
& \leq & c(x)(|s-r|+d(g(s)x, g(r)y),
\end{eqnarray*}
for all $r,s \in [-2,2]$ and every $x,y \in F$.
\end{proof}

We end this technical section with the following result.

\begin{lema}\label{lemmin}
Let $P : K \times U \to \mathbb{R}$ be a continuous function,
where $K$ is a compact space and $U$ is a topological space. Then
we have that $p :U \to \mathbb{R}$, given by
\[
p(x) = \min \{P(k,x) : k \in K\},
\]
is a continuous function.
\end{lema}
\begin{proof}
Let $x \in U$. Since $P$ is continuous, given $\varepsilon > 0$,
for each $k \in K$, there exist an open neighborhood $A_k \subset
K$ of $k$ and an open neighborhood$B_k \subset U$ of $x$ such that
$|P(l,x) -P(l,y)| < \varepsilon$, for all $l \in A_k$ and all $y
\in B_k$. Since $\{A_k : k \in K\}$ is an open cover of the
compact set $K$, there exists a finite subcover
$\{A_{k_1},\ldots,A_{k_n}\}$. Defining $B = B_{k_1} \cap \cdots
\cap B_{k_n}$, it follows that $|P(l,x) - P(l,y)| < \varepsilon$,
for all $l \in K$ and all $y \in B$. Hence we have that $|p(x) -
p(y)| < \varepsilon$, for all $y \in B$, showing that $p$ is
continuous.
\end{proof}

\section{Asymptotic behavior}

We begin this section recalling the asymptotic concepts which we
deal with in this note (see \cite{ccp}, \cite{con} and
\cite{hur}). Let $\sigma^t$ be a continuous-time ($t \in
\mathbb{R}$) or a discrete-time ($t \in \mathbb{Z}$) flow on a
metric space $(U,d)$. The $\omega$-limit set of a given point $x
\in U$ is the set of points $y \in U$ such that there exists a
sequence $t_k \to \infty$ such that $\sigma^{t_k}(x) \to y$. A
point $x \in U$ is \emph{recurrent} if it belongs to its own
$\omega$-limit set. The set of all recurrent points is called
\emph{recurrent set of $\sigma^t$} and denoted by
$\mathcal{R}(\sigma^t)$. A subset $V$ of $U$ is
\emph{$\sigma^t$-invariant} if $\sigma^t(V) = V$ for each time
$t$. The \emph{stable set of $V$}, denoted by $\mbox{st}(V)$, is
the set of all points having its $\omega$-limit set contained in
$V$. Given $x, y \in U$, $\varepsilon \in \mathcal{P}(U)$ and
$t>0$, an $(\varepsilon,t)$-chain from $x$ to $y$ is given by a
set of times $t_i > t$ and a set of points $x_i \in U$, with $x_0
= x$, $x_n = y$ and such that $d(x_{i+1}, \sigma^{t_i}(x_i)) <
\varepsilon(\sigma^{t_i}(x_i))$, where $i=1,\ldots,k$. Two given
points $x,y \in U$ are \emph{chain equivalent} if, for all
$\varepsilon \in \mathcal{P}(U)$ and $t>0$, there exist an
$(\varepsilon,t)$-chain from $x$ to $y$ and also from $y$ to $x$.
A set $\mathcal{M}$ is \emph{chain transitive} if every two points
of it are chain equivalent. A point is \emph{chain recurrent} if
it is chain equivalent with itself. The set of all chain recurrent
points is called \emph{chain recurrent set of $\sigma^t$} and
denoted by $\mathcal{R}_c(\sigma^t)$. It is easy to verify that
the chain equivalence is an equivalence relation in the chain
recurrent set. Hence the chain recurrent set can be partitioned
into chain equivalence classes which are called the \emph{chain
transitive components}. It can be proved that the recurrent and
the chain recurrent sets are $\sigma^t$-invariant as well each
chain transitive component and each stable set of a given
$\sigma^t$-invariant set.

Let $E$ be a $g$-invariant subset of $F$. Using equation
(\ref{eqgtgn}), we have that
\[
S^1 \times_g E = \bigcup_{s \in \mathbb{R}} \{\overline{s}\}
\times g(s)E
\]
is a well defined $\phi^t$-invariant subset of $S^1 \times F$. We
use this construction to relate the asymptotic behavior of the
associated skew-product flow $\phi^t$ with the asymptotic behavior
of the discrete-time flow on $F$ generated by the invertible
operator $g$. We begin relating their recurrent sets.

\begin{proposicao}\label{propR}
We have that
\[
\mathcal{R}(\phi^t) = S^1 \times_g \mathcal{R}(g).
\]
\end{proposicao}
\begin{proof}
Let $(\overline{s}, g(s)x) \in S^1 \times F$, where $s \in [0,1)$
and $x \in \mathcal{R}(g)$. By definition of $\mathcal{R}(g)$,
there exists $n_k \to \infty$ such that $g^{n_k}x \to x$. Using
equation (\ref{eqphisgsx}), we have that
\[
\phi^{n_k}(\overline{s}, g(s)x) = (\overline{s}, g(s)g^{n_k}x) \to
(\overline{s}, g(s)x),
\]
showing that $S^1 \times_g \mathcal{R}(g) \subset
\mathcal{R}(\phi^t)$. On the other hand, let $(\overline{s},
g(s)x) \in \mathcal{R}(\phi^t)$, where $s \in [0,1)$. By
definition of $\mathcal{R}(\phi^t)$, there exists $t_k \to \infty$
such that
\[
\phi^{t_k}(\overline{s}, g(s)x) \to (\overline{s}, g(s)x).
\]
By Lemma \ref{lemphisgsx}, there exist $\tau_k \in (-1,1)$ and
$n_k \in \mathbb{N}$ such that
\[
t_k = \tau_k + n_k \qquad \mbox{and} \qquad s+\tau_k \in [0,1).
\]
Thus we have that $n_k \to \infty$, since $t_k \to \infty$. Using
equation (\ref{eqphisgsx}), we have that
\[
\phi^{t_k}(\overline{s}, g(s)x) = (\overline{s+\tau_k},
g(s+\tau_k)g^{n_k} x) \to (\overline{s}, g(s)x).
\]
Hence $\overline{s+\tau_k} \to \overline{s}$ and, since $s$ and
$s+\tau_k \in [0,1)$, we have that $s+\tau_k \to s$. Therefore
\[
g^{n_k} x = g(s+\tau_k)^{-1}g(s+\tau_k)g^{n_k} x \to
g(s)^{-1}g(s)x = x,
\]
showing that $x \in \mathcal{R}(g)$ and that $\mathcal{R}(\phi^t)
\subset S^1 \times_g \mathcal{R}(g)$.
\end{proof}

Now we consider the invariant and chain transitive sets.

\begin{proposicao}\label{proptransitivo}
If $\mathcal{M}$ is invariant and chain transitive for $g$, then
$S^1 \times_g \mathcal{M}$ is chain transitive for $\phi^t$.
\end{proposicao}
\begin{proof}
First we prove that, for all $\varepsilon \in \mathcal{P}(S^1
\times F)$, $t>0$, $s \in [0,1)$ and for every $x,y \in
\mathcal{M}$, there exists an $(\varepsilon,t)$-chain from
$(\overline{0},x)$ to $(\overline{s},g(s)y)$. By Lemma
\ref{lemmin}, defining
\[
\delta(z) = C^{-1} \min \{ \varepsilon(\overline{r},g(r)z) : \, r
\in [0,1]\},
\]
we have that $\delta \in \mathcal{P}(F)$. Since $\mathcal{M}$ is
chain transitive, there exists a $(\delta,t)$-chain from $x$ to
$y$ given by $n_i > t$ and $x_i \in F$ such that
$d(x_{i+1},g^{n_i}x_i)<\delta(g^{n_i}x_i)$, where $i=1,\ldots,k$.
Putting $r = s/k < 1$ and defining
\[
t_i = n_i + r > t, \qquad \eta_i = (\overline{0},x)
\qquad\mbox{and}\qquad \eta_{i+1} = (\overline{ir},g(ir)x_{i+1}),
\]
where $i=1,\ldots,k$, we claim that this provides an
$(\varepsilon,t)$-chain from $(\overline{0},x)$ to
$(\overline{s},g(s)y)$. In fact, since $\phi^{t_i}(\eta_i) =
(\overline{ir},g(ir)g^{n_i}x_i)$ and $ir \in [0,1]$, using
inequality (\ref{eqlipg}), we have that
\[
\overline{d}(\eta_{i+1},\phi^{t_i}(\eta_i)) = d(g(ir)x_{i+1},
g(ir)g^{n_i}x_i) \leq Cd(x_{i+1}, g^{n_i}x_i) < C\delta(
g^{n_i}x_i)
\]
and hence
\[
\overline{d}(\eta_{i+1},\phi^{t_i}(\eta_i)) <
\varepsilon(\overline{ir},g(ir)g^{n_i}x_i) =
\varepsilon(\phi^{t_i}(\eta_i)).
\]
We also have that $\eta_{k+1} = (\overline{s},g(s)y)$, since $kr =
s$ and $x_{k+1} = y$.

Now we prove that, for all $\varepsilon \in \mathcal{P}(S^1 \times
F)$, $t>0$, for each $u,v \in [0,1)$ and for every $x,y \in
\mathcal{M}$, there exists an $(\varepsilon,t)$-chain from
$(\overline{u},g(u)x)$ to $(\overline{v},g(v)y)$. We denote
$\delta = C^{-1}\varepsilon \circ \phi^u \in \mathcal{P}(S^1
\times F)$. If $v \geq u$, we have that $v - u \in [0,1)$. Using
the first part of the proof, there exists a $(\delta,t)$-chain
from $(\overline{0},x)$ to $(\overline{v-u},g(v-u)y)$, denoted by
$t_i
> t$ and $\eta_i \in S^1 \times F$, with
$\overline{d}(\eta_{i+1}, \phi^{t_i}(\eta_i)) <
\delta(\phi^{t_i}(\eta_i))$, where $i=1,\ldots,k$. Taking the same
times $t_i > t$ and putting $\xi_i = \phi^u(\eta_i)$, we get an
$(\varepsilon,t)$-chain from $(\overline{u},g(u)x)$ to
$(\overline{v},g(v)y)$. In fact, using inequality
(\ref{eqlipphi}), we have that
\[
\overline{d}(\xi_{i+1}, \phi^{t_i}(\xi_i)) =
\overline{d}(\phi^u(\eta_{i+1}), \phi^u(\phi^{t_i}(\eta_i))) \leq
C\overline{d}(\eta_{i+1}, \phi^{t_i}(\eta_i)) <
\varepsilon(\phi^{t_i}(\xi_i))
\]
and that
\[
\xi_1 = \phi^u(\overline{0},x) = (\overline{u},g(u)x)
\quad\mbox{and}\quad \xi_{k+1} = \phi^u(\overline{v-u},g(v-u)y) =
(\overline{v},g(v)y).
\]
If $v < u$, we have that $1 + v - u \in [0,1)$. Using again the
first part of the proof, there exists a $(\delta,t)$-chain from
$(\overline{0},x)$ to $(\overline{1+v-u},g(1+v-u)g^{-1}y)$,
denoted by $t_i > t$ and $\eta_i \in S^1 \times F$, where
$i=1,\ldots,k$. Arguing exactly as before, we obtain an
$(\varepsilon,t)$-chain from $(\overline{u},g(u)x)$ to
$(\overline{v},g(v)y)$, if we take the same times $t_i > t$ and
put $\xi_i = \phi^u(\eta_i)$, since
\[
\phi^u(\overline{1+v-u},g(1+v-u)g^{-1}y) =
(\overline{1+v},g(1+v)g^{-1}y) = (\overline{v},g(v)y).
\]
\end{proof}

In the following, we provide the relation between the chain
recurrent sets.

\begin{teorema}\label{teoRc}
We have that
\[
\mathcal{R}_c(\phi^t) = S^1 \times_g \mathcal{R}_c(g).
\]
\end{teorema}
\begin{proof}
By Theorem 3 of \cite{ccp}, we have that $\mathcal{R}_c(\phi^t) =
\mathcal{R}_c(\phi^1)$. Hence we just need to prove that
$\mathcal{R}_c(\phi^1) = S^1 \times_g \mathcal{R}_c(g)$. Since
every chain recurrent set is partitioned in its chain transitive
components, using Proposition \ref{proptransitivo}, it follows
that $S^1 \times_g \mathcal{R}_c(g) \subset
\mathcal{R}_c(\phi^1)$. Let $(\overline{s},g(s)x) \in
\mathcal{R}_c(\phi^1)$, where $s \in [0,1)$ and $x \in F$. We
claim that $x \in \mathcal{R}_c(g)$. Let $\varepsilon \in
\mathcal{P}(F)$ and $n \in \mathbb{N}$. By Lemma \ref{lemmin},
defining
\[
\delta(\overline{r},y) = 2^{-1} \min \{
c(g(t)^{-1}y)^{-1}\varepsilon(g(t)^{-1}y) : \, t \in [0,1]\},
\]
it follows that $\delta \in \mathcal{P}(S^1 \times F)$. Hence
there exists an $(\delta,n)$-chain from $(\overline{s},g(s)x)$ to
itself, denoted by $m_i > 2n$ and $\xi_i =
(\overline{s_i},g(s_i)x_i) \in S^1 \times F$, where $s_i \in
[0,1)$ and $\overline{d}(\xi_{i+1}, \phi^{m_i}(\xi_i)) <
\delta(\phi^{m_i}(\xi_i))$, where $i=1,\ldots,k$. Hence we have
that
\[
\min\{|s_{i+1}-s_i|,1-|s_{i+1}-s_i|\} < \delta(\phi^{m_i}(\xi_i)).
\]
Thus we get the following three possibilities and the
corresponding choices
\[
\begin{array}{lll}
(1)   & |s_{i+1}-s_i| <  \delta(\phi^{m_i}(\xi_i)),       & n_i = m_i \\
(2)  & (1 + s_i) - s_{i+1} <  \delta(\phi^{m_i}(\xi_i)), & n_i = m_i-1 \\
(3) & s_{i+1} - (s_i-1) <  \delta(\phi^{m_i}(\xi_i)),   & n_i = m_i+1 \\
\end{array}
\]
In each possibility, we have that $n_i > n$ and, using inequality
(\ref{eqcx}), we get that
\[
d(x_{i+1},g^{n_i}x_i) \leq 2c(g^{n_i}x_i)\delta(\phi^{m_i}(\xi_i))
\leq \varepsilon(g^{n_i}x_i),
\]
showing that $x \in \mathcal{R}_c(g)$, since $x_0 = x = x_{k+1}$.
\end{proof}

\begin{corolario}\label{corM}
The map $E \mapsto S^1 \times_g E$, where $E \subset F$ is
$g$-invariant, is a bijection between chain transitive components
of $g$ and of $\phi^t$
\end{corolario}
\begin{proof}
By Proposition \ref{proptransitivo}, if $\mathcal{M}$ is a chain
transitive component of $g$, then $S^1 \times_g \mathcal{M}$ is
chain transitive for $\phi^t$ and hence is contained in a chain
transitive component of $\phi^t$. The result follows, since
$\mathcal{R}_c(\phi^t) = S^1 \times_g \mathcal{R}_c(g)$ and since
every chain recurrent set is partitioned in its chain transitive
components.
\end{proof}

We end this note relating the stable sets of chain transitive
components.

\begin{proposicao}\label{propestavel}
If $\mathcal{M}$ is a chain transitive component of $g$, then
\[
\emph{st}(S^1 \times_g \mathcal{M}) = S^1 \times_g
\emph{st}(\mathcal{M}).
\]
\end{proposicao}
\begin{proof}
Since the $\omega$-limit sets are chain transitive, a point
belongs to the stable set of a given chain component if and only
if its omega limit set intersects the chain component. For every
$x \in \mbox{st}(\mathcal{M})$, there exists $n_k \to \infty$ and
$y \in \mathcal{M}$ such that $g^{n_k}x \to y$. In this case, for
each $s \in [0,1)$, we have that
\[
\phi^{n_k}(\overline{s},g(s)x) = (\overline{s},g(s)g^{n_k}x) \to
(\overline{s},g(s)y),
\]
showing that $(\overline{s},g(s)x) \in \mbox{st}(S^1 \times_g
\mathcal{M})$ and that $S^1 \times_g \mbox{st}(\mathcal{M})
\subset \mbox{st}(S^1 \times_g \mathcal{M})$. For the other
inclusion, let $(\overline{s},g(s)x) \in \mbox{st}(S^1 \times_g
\mathcal{M})$, where $s \in [0,1)$. Thus there exists $t_k \to
\infty$, $y \in \mathcal{M}$ and $r \in [0,1)$ such that
\[
\phi^{t_k}(\overline{s},g(s)x) \to (\overline{r},g(r)y).
\]
Hence we have that $\phi^{t_k + s - r}(\overline{s},g(s)x) \to
(\overline{s},g(s)y)$. By Lemma \ref{lemphisgsx}, there exist
$\tau_k \in (-1,1)$ and $n_k \in \mathbb{N}$ such that
\[
t_k + s - r = \tau_k + n_k \qquad \mbox{and} \qquad s+\tau_k \in
[0,1).
\]
Thus we have that $n_k \to \infty$, since $t_k \to \infty$. Using
equation (\ref{eqphisgsx}), we have that
\[
\phi^{t_k + s - r}(\overline{s}, g(s)x) = (\overline{s+\tau_k},
g(s+\tau_k)g^{n_k} x) \to (\overline{s}, g(s)y).
\]
Hence $\overline{s+\tau_k} \to \overline{s}$ and, since $s$ and
$s+\tau_k \in [0,1)$, we have that $s+\tau_k \to s$. Therefore
\[
g^{n_k} x = g(s+\tau_k)^{-1}g(s+\tau_k)g^{n_k} x \to
g(s)^{-1}g(s)y = y,
\]
showing that $x \in \mbox{st}(\mathcal{M})$ and that
$\mbox{st}(S^1 \times_g \mathcal{M}) \subset S^1 \times_g
\mbox{st}(\mathcal{M})$.
\end{proof}

\begin{remark}
\emph{All the results presented in this section remain true if we
replace the Banach space $F$ by an arbitrary compact metrizable
fiber $\mathbb{F}$ where the topological group of bounded
invertible linear operators acts continuously. In fact, in the
definition of chains, we can replace the positive continuous
functions by positive constant functions, once every positive
continuous function defined on a compact space has a positive
minimum. Hence we can replace the inequalities (\ref{eqlipg}),
(\ref{eqlipphi}) and (\ref{eqcx}) by the uniform continuity,
respectively, of the map $(s,x) \mapsto g(s)x$ on $[0,1] \times
\mathbb{F}$, of the map $\phi^u$ on $\mathbb{F}$ and of the map
$(s,x) \mapsto g(s)^{-1}x$ on $[-2,2] \times \mathbb{F}$. }
\end{remark}

\end{document}